\documentclass[10pt]{amsart}
\usepackage{amssymb}
\usepackage{amscd}
\usepackage{mathrsfs}
\usepackage[all]{xy}
\usepackage{ulem}
\usepackage{hyperref}

\theoremstyle{definition}
\newtheorem{thm}{Theorem}[section]
\newtheorem{lem}[thm]{Lemma}
\newtheorem{prp}[thm]{Proposition}
\newtheorem{dfn}[thm]{Definition}
\newtheorem{cor}[thm]{Corollary}

\newtheorem{rmk}[thm]{Remark}

\newtheorem{exa}[thm]{Example}

\renewcommand{\qed}{\rule{0.4em}{2ex}}

\newcommand{\beq}{\begin{equation}}
\newcommand{\eeq}{\end{equation}}
\newcommand{\beqr}{\begin{eqnarray*}}
\newcommand{\eeqr}{\end{eqnarray*}}
\newcommand{\bal}{\begin{align*}}
\newcommand{\eal}{\end{align*}}
\newcommand{\bei}{\begin{itemize}}
\newcommand{\eei}{\end{itemize}}

\newcommand{\C}{{\mathbb{C}}}
\newcommand{\N}{{\mathbb{N}}}
\newcommand{\M}{{\mathbb{M}}}

\newcommand{\bra}[1]{\left\langle #1 \right|}
\newcommand{\ket}[1]{\left| #1 \right\rangle}

\title[Subspaces with Schmidt rank bounded below]{The order-$n$ minors of certain  $(n+k) \times n$ matrices}

\author{Priyabrata Bag}

\address{School of Mathematical Sciences\\
	Narsee Monjee Institute of Management Studies\\
	Vile Parle (W), Mumbai, Maharashtra 400056, India}

\email[]{priyabrata.bag@gmail.com}

\author{Santanu Dey}

\address{Department of Mathematics\\
Indian Institute of Technology Bombay\\
Mumbai, Maharshtra 400076, India}

\email[]{santanudey@iitb.ac.in}

\author{Masaru Nagisa}

\address{Department of Mathematics and Informatics \\Chiba University \\
Yayoi 1-33, Chiba, 263-8522, Japan}

\email[]{nagisa@math.s.chiba-u.ac.jp}

\author {Hiroyuki Osaka}

\address{ Department of Mathematical Sciences\\ Ritsumeikan University\\ Kusatsu, Shiga, 525-8577  Japan}

\email[]{osaka@se.ritsumei.ac.jp}

\date{\today} %\textcolor{magenta}{\; (The following footnotes need to be modified.)}}
\keywords{Minor, Schmidt rank, Schmidt number}
\subjclass[2010]{15A15; 81P40; 81P68; 15A03.}

\begin{document}

%\maketitle

\begin{abstract}

We determine sufficient conditions
 for certain classes of $(n+k) \times n$  matrices $E$ to have all order-$n$ minors to be nonzero. For a special class of $(n+1) \times n$ matrices $E,$  we give the formula for the order-$n$ minors. As an application we construct subspaces of $\C^m \otimes \C^n$ of maximal dimension, which does not contain any vector of Schmidt rank less than $k$ and which has a basis of Schmidt rank $k$ for $k=2,3,4$. 
%Let $E$ be a $(n+k) \times n$ matrix, and  We determine sufficient conditions for invertibility of  
%$E^{i_1, i_2, \dots, i_k}$ for certain classes of matrices $E$ for all distinct combinations of $i_k \in \{1,2, \dots,n+k\}$. For %a special class of $(n+1) \times n$ -matrices $E,$  we give the determinant formula for $E^{i_1}$. As an application we %construct of subspaces of $C^m \otimes \C^n$ of maximal dimension, which does not contain any vector of Schmidt rank %less than 4. 

\end{abstract}

\maketitle

%%%%%%%%%%%%%%%%%%%%%%%%%%%%%%%%%%%%%%%%%%%%%%%%%%%%%%%%%%%%%%%%%%%%%%%%%%%%%%%
%%%%%%%%%%%%%%%%%%%%%%%%%%%%%%%%%%%%%%%%%%%%%%%%%%%%%%%
\section{Introduction}

In this article we investigate conditions under which certain 
$(n+k) \times n$ matrices $E$ have all order-$n$ minors to be nonzero. Using this we  conclude that for such $(n+k) \times n$ matrix $E,$ any linear combination of the columns of $E$ will have at least $k+1$ nonzero entries. This property has an application to the computation of Schmidt rank.

For any $n \in \N$ and numbers $a$ and $b$ in $\C,$ set the $(n+1)\times n$ matrix $E_n(a,b) \in \M_{n+1,n}(\C)$ as follows:
\begin{equation} \label{EN}  E_n(a,b) = b \sum_{j=2}^n |j-1 \rangle \langle j | -a \sum_{j=1}^n |j\rangle \langle j| 
                      + a \sum_{j=1}^{n} |j+1\rangle \langle j| -b \sum_{j=1}^{n-1} |j+2\rangle \langle j| , \end{equation}
where $|1\rangle,\ldots, |n+1\rangle$ are the standard basis column vectors of $\C^{n+1}$, as column vectors, and $\langle 1|,\ldots, \langle n| $ are
the transposed standard basis column vectors of $\C^n$.

%For any $n \in \N$ and positive numbers $a$ and $b$ in $\R,$ set the $(n+1)\times n$ matrix 
%\begin{equation} \label{EN}
%E_n(a, b) = 
%\left(
%\begin{array}{ccccccc}
%-a&b&0&0&0&\cdots&0\\
%a&-a&b&0&0&\cdots&0\\
%-b&a&-a&b&0&\cdots&0\\
%0&-b&a&-a&b&\ddots&\vdots\\
%\vdots&\ddots&\ddots&\ddots&\ddots&\ddots&0\\
%\vdots&&\ddots&\ddots&\ddots&\ddots&b\\
%0&\cdots&\cdots&0&-b&a&-a\\
%0&\cdots&\cdots&\cdots&0&-b&a
%\end{array}
%\right).
%\end{equation} 
%
%\textcolor{red}{The positivity of $a$ and $b$ are needed ? It seems that the positivity is not needed.}

In  Section  \ref{MR} we obtain an expression for order-$n$ minors of $E_n(a,b)$ in terms of determinants of certain submatrices of $E_n(a,b).$ By exploiting a recurrence relation and the characteristic polynomial of a certain matrix, we derive formulae for order-$n$ minors of $E_n(a,b).$  For some examples of $E_n(a,b),$ the order-$n$ minors are computed in Section \ref{EXA}.   In Section \ref{IoE} we take an approach based on system of equations, and  conclude that for $|a|\geq 5,$ all the order-$n$ minors of $E_n(a,1)$ are nonzero. We also analyse another class of $(n+2)\times n$ matrices $G_n(a,1)$ using the same approach.

Let ${\mathcal H}$ denote the bipartite Hilbert space $\C^m\otimes\C^n$. By Schmidt decomposition theorem
(cf. Ref.~ \cite{KRP 2004}), any pure state $\ket{\psi}\in {\mathcal H}$ can be written as
\begin{equation}
 \ket{\psi}=\sum_{j=1}^k\alpha_j\ket{u_j}\otimes\ket{v_j} \label{sd}
\end{equation}
for some $k\leq\min\{m,n\}$, where $\{\ket{u_j}:1\leq j\leq k\}$ and
$\{\ket{v_j}:1\leq j\leq k\}$ are orthonormal sets in $\C^m$ and $\C^n$ respectively,
and $\alpha_j$'s are positive real numbers satisfying $\sum_j\alpha_j^2=1$.
\begin{dfn}
 In the Schmidt decomposition \eqref{sd} of a pure bipartite state $\ket{\psi}$ the minimum number
 of terms required in the summation is known as the Schmidt rank of $\ket{\psi}$, and it is
 denoted by $SR(\ket{\psi})$.
\end{dfn}

 As an application of the results of the first four sections, we construct subspaces $\mathcal{T}$ of dimension $(m-k+1)(n-k+1)$ of bipartite finite dimensional Hilbert space $\C^m \otimes \C^n$ in Section \ref{SMDBSR} such that any vector in $\mathcal{T}$ has Schmidt rank greater  than or equal to $k$ where $k=2,3$ and $4.$ In Ref.~\cite{CMW 2008}, it was proved that for a bipartite
system $\C^m \otimes \C^n,$ the dimension of any subspace of Schmidt rank greater than or equal
to $k$ is bounded above by $(m-k+1)(n-k+1).$  Unlike
Ref.~\cite{CMW 2008}, the 
subspaces $\mathcal{T}$ of $\C^m \otimes \C^n$ that we construct also have  bases consisting of elements of Schmidt rank $k.$
For the case when a subspace of $\C^m\otimes\C^n$ is of Schmidt rank  greater than or equal to 2
(that is, the subspace does not contain any product vector), the maximum dimension of that subspace
is  $(m-1)(n-1),$ and this was first proved  in Ref.~\cite{KRP 2004} and Ref.~\cite{Wal 2002}
(cf. Ref.~\cite{Bhat 2006}).

In the bipartite Hilbert space $\C^m\otimes\C^n,$ for any $1 \leq r \leq\min\{m,n\},$
there is at least some state $\ket{\psi}$ with $SR(\ket{\psi})=r.$ Any state $\rho$ on a finite
dimensional Hilbert space ${\mathcal H}$ can be written as
\begin{equation}
 \rho=\sum_jp_j\ket{\psi_j}\bra{\psi_j}, \label{spd}
\end{equation}
where  $\ket{\psi_j}$'s are pure states in ${\mathcal H}$ and $\{p_j\}$ forms a probability
distribution. The following notion was introduced in Ref.~\cite{TH 2000}:
\begin{dfn}
 The Schmidt number of a state $\rho$ on a bipartite finite dimensional Hilbert space ${\mathcal H}$ is defined
 to be the least natural number $k$ such that $\rho$ has a decomposition of the form given in 
 \eqref{spd} with $SR(\ket{\psi_j})\leq k$ for all $j$. The Schmidt number of $\rho$ is
 denoted by $SN(\rho)$.
\end{dfn}

Schmidt number of a state on a bipartite Hilbert space is a measure of entanglement. Entanglement is the key property of quantum systems which is responsible for the higher
efficiency of quantum computation and tasks like teleportation, super-dense coding, etc
(cf. Ref.~\cite{HHHH 2009}). The  following proposition establishes an important relation between  Schmidt number of a state and the lower bound of  Schmidt rank of any non-zero vector in the supporting subspace of the state:

%\textcolor{red}{
%The following should be well-known.
%\vskip 2mm
\begin{prp}
 Let $\mathcal{S}$ be a subspace of ${\mathcal H}=\C^m\otimes\C^n$ which does not contain any non-zero vector of Schmidt rank
 lesser or equal to $k$. Then any state $\rho$ supported on $\mathcal{S}$ has Schmidt number at least $k+1$.
\end{prp}
%}

\begin{proof}
Let $\rho$ be a state with Schmidt number, $SN(\rho) = r \leq k$. 
So, $\rho$ can be written
as 
$$
\rho =\sum_jp_j |v_j\rangle \langle v_j |,
$$
where $SR(|v_j\rangle) \leq k$
for all $j$ and $\{p_j\}$ forms a probability distribution. 
Let $|\psi\rangle \in \mathrm{Ker}(\rho)$. Then we have
$$
0 = \langle\psi|\rho| \psi \rangle =\sum_jp_j|\langle v_j|\psi\rangle|^2.
$$
This means $|v_j\rangle \in \mathrm{Ker}(\rho)^\perp = \mathrm{Range}(\rho)$ for all $j$.

If such a $\rho$ is supported on $\mathcal{S}$, then the above statement gives a contradiction to the fact that $\mathcal{S}$ does not contain any vector of Schmidt rank lesser or equal to $k$.
This completes the proof.
\end{proof}

%%%%%%%%%%%%%%%%%%%%%%%%%%%%%%%%%%%%%%%%%%%%%%%%%%%%%%
\section{Main Results}\label{MR}

\begin{lem}\label{lem:$D_n$}
For any $n \in \N$ and numbers $a$ and $b$ in $\C$ set the $n \times n$ matrix
$$
D_n(a, b) = b \sum_{j=2}^n |j-1 \rangle \langle j | -a \sum_{j=1}^n |j\rangle \langle j| 
                      + a \sum_{j=1}^{n-1} |j+1\rangle \langle j| -b \sum_{j=1}^{n-2} |j+2\rangle \langle j| \in \M_n(\C) .
$$
Then we have $|D_n(a,b)| = -a |D_{n-1}(a,b)| - ab |D_{n-2}(a, b)| - b^3 |D_{n-3}(a, b)|$ \
 $(n \geq 2)$, where  $|D_{-1}(a,b)|=0$ and $|D_0(a, b)| = 1$. Note that $|D_1(a,b)| = -a$.
 %{We may need to specify $|D_1(a,b)|$ as well.}
\end{lem}

\vskip 3mm

We also define the $n \times n$ matrix $F_n(a,b)\in \M_n(\C)$ as follows:
\begin{align*}
F_n(a, b) & = b \sum_{j=3}^n |j-2 \rangle \langle j | -a \sum_{j=2}^n |j-1\rangle \langle j| 
                      + a \sum_{j=1}^{n} |j\rangle \langle j| -b \sum_{j=1}^{n-1} |j+1\rangle \langle j|  \\
           & = -b \sum_{j=2}^n |j \rangle \langle j-1 | +a \sum_{j=1}^n |j\rangle \langle j| 
                      - a \sum_{j=1}^{n-1} |j\rangle \langle j+1| +b \sum_{j=1}^{n-2} |j\rangle \langle j+2|  \\
           & =  D_n(-a,-b)^t.
\end{align*}
Then we have $|F_n(a,b)| = (-1)^n |D_n(a,b)| $ \ $(n \geq -1)$.

\vskip 3mm

For an $n\times n$ matrix $A =\sum_{i,j=1}^n a_{ij} |i\rangle \langle j| \in \M_n(\C)$, 
we define $A^\uparrow$ and $A^\downarrow$ in $\M_{n-1}(\C)$ as follows:
\[  A^\uparrow = \sum_{i,j=1}^{n-1} a_{i,j}|i\rangle \langle j| \text{ and }
     A^\downarrow = \sum_{i,j=2}^n a_{i,j} |i\rangle \langle j| , \]
if $n\ge 2$.

\vskip 2mm

For numbers $a,b \in \C$, $A=\sum_{i,j=1}^n a_{i,j} |i\rangle \langle j| \in \M_n(\C)$ and
$B=\sum_{i,j=1}^m b_{i,j} |i\rangle \langle j| \in \M_m(\C)$, 
we define the matrix $C=\sum_{i,j=1}^{n+m}c_{i,j} |i\rangle \langle j|\in \M_{n+m}(\C)$ as follows:
\[   c_{ij} =\begin{cases}  a_{ij} \; & \text{ if } 1\le i,j\le n,  \\   b_{i-n,j-n}  & \text{ if } n+1 \le i,j \le n+m \\
                    a & \text{ if } (i,j)=(n+1,n) \\  b & \text{ if } (i,j)=(n,n+1) \\  0 & \text{ otherwise }
                   \end{cases}.\]
 We denote this matrix $C$ by $A {}_a\!\!\oplus_b B$.

\begin{lem}\label{lem:$A+B$}
In this setting, we have
\[   |A {}_a\!\!\oplus_b B|  = |A| |B| - ab |A^\uparrow| |B^\downarrow | , \]
where we set $|A^\uparrow|=|A^\downarrow|=1$ if $A\in \M_1(\C)$.
\end{lem}
\begin{proof}
We denote $A {}_a\!\!\oplus_b B =C=\sum_{i,j=1}^{n+m}c_{i,j} |i\rangle \langle j|$.
For $k\le l$, we denote $S(k,l)$ the set of all permutations on $\{k, k+1,\ldots,l -1,l \}$.
Let $\sigma \in S(1,n+m)$ and ${\rm sgn}(\sigma)$ denote the signature of permutation $\sigma$.
Then $\sigma$ has the form which is one of the following type:
\begin{enumerate}
\item[(1)] there exist $\sigma_1\in S(1,n)$ and $\sigma_2\in S(n+1,n+m)$ such that
\[  \sigma(i) =\sigma_1(i) \; (1\le i \le n) \text{ and } \sigma(i) = \sigma_2(i) \; (n+1\le i \le n+m) . \]
\item[(2)] there exist $\tau_1 \in S(1,n-1)$ and $\tau_2\in S(n+2,n+m)$ such that
\[  \sigma(n)=n+1, \sigma(n+1)=n, \sigma(i) =\tau_1(i) \; (1\le i \le n-1) \text{ and } \sigma(i) = \tau_2(i) \; (n+2\le i \le n+m) . \]
\item[(3)] $\sigma$ has the form of neither type (1) nor type (2).
\end{enumerate} 
Then we have
\[   {\rm sgn}(\sigma) \prod_{i=1}^{n+m}c_{i,\sigma(i)} =
              \begin{cases}
                  {\rm sgn}(\sigma_1){\rm sgn}(\sigma_2)\prod_{i=1}^n c_{i,\sigma_1(i)}
                         \prod_{i=n+1}^{n+m} c_{i,\sigma_2(i)}    & \text{ case}(1) \\
                  -ab\, {\rm sgn}(\tau_1) {\rm sgn}(\tau_2) \prod_{i=1}^{n-1} c_{i,\tau_1(i)}
                        \prod_{i=n+2}^{n+m} c_{i, \tau_2(i)}  & \text{ case}(2) \\
                 0   & \text{ case}(3)
               \end{cases}.
\]
This implies that
\begin{align*}
  |C| & = \sum_{\sigma\in S(1,n+m)}{\rm sgn}(\sigma) \prod_{i=1}^{n+m} c_{i,\sigma(i)}  \\
    &  = \sum_{\sigma_1\in S(1,n),\sigma_2\in S(n+1,n+m)}{\rm sgn}(\sigma_1) {\rm sgn}(\sigma_2) 
            \prod_{i=1}^n c_{i,\sigma_1(i)} \prod_{i=n+1}^{n+m}c_{i,\sigma_2(i)}  \\
    & \qquad  - ab \sum_{\tau_1\in S(1,n-1), \tau_2\in S(n+2,n+m)} {\rm sgn}(\tau_1){\rm sgn}(\tau_2)
            \prod_{i=1}^{n-1} c_{i,\tau_1(i)} \prod_{i=n+2}^{n+m}c_{i,\tau_2(i)}  \\
    & = \left( \sum_{\sigma_1\in S(1,n)}{\rm sgn}(\sigma_1) \prod_{i=1}^n a_{i,\sigma_1(i)} \right)
            \left( \sum_{\sigma_2\in S(1,m)}{\rm sgn}(\sigma_2) \prod_{i=1}^{m}b_{i,\sigma_2(i)} \right)  \\
    & \qquad  - ab \left( \sum_{\tau_1\in S(1,n-1)} {\rm sgn}(\tau_1) \prod_{i=1}^{n-1} a_{i,\tau_1(i)} \right)
            \left( \sum_{\tau_2\in S(2,m)} {\rm sgn}(\tau_2) \prod_{i=2}^{m}b_{i,\tau_2(i)} \right)  \\
    & =  |A| |B| -ab |A^\uparrow| |B^\downarrow|  .
\end{align*}
\end{proof}

\vskip 3mm

\begin{prp}\label{prp:$E_n$}
For $1 \leq k \leq n + 1$ and $a$ and $b$ in $\C$,  let $E_n^k(a,b)$ be a matrix which is obtained by deleting the $k$-th row of $E_n(a,b)$ which is defined in equation \eqref{EN}. 
Then we have 
\begin{align*}
|E_n^k(a,b)| & = |D_{k-1}(a,b)||F_{n-(k-1)}(a,b)| + |D_{k-2}(a,b)||F_{n-k}(a,b)|b^2  \\
                & = (-1)^{n-k+1} (|D_{k-1}(a,b)| |D_{n-k+1}(a,b)| - |D_{k-2}(a,b)| |D_{n-k}(a,b)|b^2).
\end{align*} 
\end{prp}

%\begin{prp}\label{prp:$E_n$}
%For $1 \leq k \leq n + 1$ and positive numbers $a$ and $b$  let $E_n^k(a,b)$ be a matrix which is obtained by deleting $k$ row of 
%$E_n(a,b)$ which is defined in equation \eqref{EN}. 
%Then we have
%$$
%E_n^k(a,b)| = |D_{k-1}(a,b)||F_{n-(k-1)}(a,b)| + |D_{k-2}(a,b)||F_{n-k}(a,b)|b^2,$$
%$(|(D_0 (a,b)| = 1$, $|D_{-1}(a,b)|  =0$, $|F_0(a,b)| = 1$, $|F_{-1}(a,b)| = 0)$
%\end{prp}

\begin{proof}
Since $E_n^1(a,b) = F_n(a,b)$, $E_n^{n+1}(a,b) = D_n(a,b)$, $|D_0(a,b)|=|F_0(a,b)|=1$ and
$|D_{-1}(a,b)|=|F_{-1}(a,b)|=0$, we have
\begin{align*}
  |E_n^1(a,b)| & = |F_n(a,b)| = |D_0(a,b)| |F_n(a,b)| + |D_{-1}(a,b)| |F_{n-1}(a,b)| b^2 \\
\intertext{and}
  |E_n^{n+1}(a,b)| & = |D_n(a,b)| = |D_n(a,b)| |F_0(a,b)| + |D_{n-1}(a,b)| |F_{-1}(a,b)|b^2.
\end{align*}

In the case $2\le k \le n$, we have $E_n^k(a,b) = D_{k-1}(a,b) {}_{(-b)}\oplus_b F_{n-k+1}(a,b)$.
It implies that, by Lemma~\ref{lem:$A+B$}, 
\begin{align*}
    |E_n^k(a,b)| & = |D_{k-1}(a,b)| |F_{n-k+1}(a,b)| +b^2 |D_{k-1}(a,b)^\uparrow| |F_{n-k+1}(a,b)^\downarrow| \\
                            & = |D_{k-1}(a,b)| |F_{n-k+1}(a,b)| + |D_{k-2}(a,b)| |F_{n-k}(a,b)| b^2.
\end{align*} 

\end{proof}

%\vskip 3mm
%
%Note that for each $n \in \N$ $|F_n(a,b)| = (-1)^n|D_n(a,b)|$.

\vskip 3mm

If $b= 0$ and $1 \leq k \leq n+1$, 
\begin{align*}
|E_n^k(a,b)| &= (-1)^{n-k+1}|D_{k-1}(a,0)||D_{n-k+1}(a,0)|\\
&= (-1)^{n-k+1}(-a)^{k-1}(-a)^{n-k+1}\\
&= (-1)^{k-1}a^n.
\end{align*}

If $b\not=0$, since $|E_n^k(a,b)| = b^{n}\left|E_n^k\left(\dfrac{a}{b},1\right)\right|$, we may assume that $b = 1$.

\vskip 3mm

After this we write $E_n(a,1) = E_n(a)$ and $D_n(a, 1) = D_n(a)$.

\vskip 3mm

\begin{thm}\label{thm:$E_n$}
For  $n \in \N$ and 
%% a positive number 
$a \in \C$,  let $d_{-1} = 0$, $d_0 = 1$ and $d_n = |D_n(a)|$.
 Then for $1 \leq k \leq n + 1$ we have 

\begin{align*}
|E_n^k(a)| & = (-1)^{n-k+1}(d_{k-1}d_{n-k+1} - d_{k-2}d_{n-k})  \\
                   & =(-1)^{k+1} |E_n^{n-k+2}(a)|\\
                   &= (-1)^n|E_n^{n-k+2}(a)|
\end{align*}

and 
$$
  d_k = \frac{1}{k!} \frac{d^k}{dx^k}\left.\left( \frac{1}{1+ax+ax^2+x^3}\right) \right|_{x=0} .
$$
 Moreover, if $x^3 + ax^2 + ax + 1 = 0$ has $3$ different solutions $\alpha, \beta, \gamma,$ then  
we have for $1 \leq k$
$$
d_k = \dfrac{1}{\alpha^{k+1}(\alpha-\beta)(\gamma - \alpha)} + \dfrac{1}{\beta^{k+1}(\alpha - \beta)(\beta-\gamma)}
+ \dfrac{1}{\gamma^{k+1}(\gamma - \alpha)(\beta - \gamma)}.
$$
\end{thm}

%\begin{thm}
%For  $n \in \N\cup \{-1, 0\} $ and a positive number $a \in \R$  let $d_{-1} = 0$, $d_0 = 1$ and $d_n = |D_n(a,1)|$. Then for $1 \leq k \leq n + 1$ we have 
%$$
%|E_n^k(a,1)| = (-1)^{n-k+1}(d_{k-1}d_{n-k+1} - d_{k-2}d_{n-k}).
%$$ 
%Moreover, if $x^3 + ax^2 + ax + 1 = 0$ has $3$ different solutions $\alpha, \beta, \gamma,$ then  
%we have for $1 \leq k$
%$$
%d_k = \dfrac{1}{\alpha^{k+1}(\alpha-\beta)(\gamma - \alpha)} + \dfrac{1}{\beta^{k+1}(\alpha - \beta)(\beta-\gamma)}
%+ \dfrac{1}{\gamma^{k+1}(\gamma - \alpha)(\beta - \gamma)}.
%$$
%\end{thm}

\begin{proof}

%Since $|F_n(a,1)| = (-1)^n|D_n(a,1)|$, we have by Proposition~\ref{prp:$E_n$}
By  Proposition~\ref{prp:$E_n$}
we have
\begin{align*}
|E_n^k(a)| &= (-1)^{n-k+1}d_{k-1}d_{n-k+1} + (-1)^{n-k}d_{k-2}d_{n-k}\\
&= (-1)^{n-k+1}(d_{k-1}d_{n-k+1} - d_{k-2}d_{n-k}).
\end{align*}

On the contrary, $(-1)^n|E_n^{n-k+2}(a)| = (-1)^n(-1)^{k-1}(d_{n-k+1}d_{k-1} - d_{n-k}d_{k-2}) = (-1)^{n-k+1}(d_{k-1}d_{n-k+1} - d_{k-2}d_{n-k})$.

Define the formal power series $f(x) =\sum\limits_{k=0}^\infty d_k x^k$, where the sequence $\{d_n\}_{n=0}^\infty$ satisfies the relation
$d_n +ad_{n-1} +ad_{n-2} + d_{n-3}=0 \ (n \geq 2)$.
Since $(1+ax+ax^2+x^3)f(x) = d_0 + (d_1+ad_0)x + (d_2+ad_1+ad_0)x^2 + \sum\limits_{k=3}^\infty(d_k+ad_{k-1}+ad_{k-2}+d_{k-3})x^k=1$,
$f(x) = \dfrac{1}{1+ax+ax^2+x^3}$ is analytic at a neighborhood of $0$.
By Maclaurin's expansion of $f$, we have
\[d_k = \frac{1}{k!} \frac{d^k}{dx^k}\left.\left( \frac{1}{1+ax+ax^2+x^3}\right) \right|_{x=0} .  \]

%By Lemma~\ref{lem:$D_n$}
%we have 
%$$d_n = -ad_{n-1} - ad_{n-2}-d_{n-3}.
%$$
%
%Then we have 
%\begin{align*}
%\left(
%\begin{array}{c}
%d_n\\
%d_{n-1}\\
%d_{n-2}
%\end{array}
%\right)
%&=
%\left(\begin{array}{ccc}
%-a&-a&-1\\
%1&0&0\\
%0&1&0
%\end{array}
%\right)
%\left(
%\begin{array}{c}
%d_{n-1}\\
%d_{n-2}\\
%d_{n-3}
%\end{array}
%\right)
%\\
%&=
%\left(\begin{array}{ccc}
%-a&-a&-1\\
%1&0&0\\
%0&1&0
%\end{array}
%\right)^{n-2}
%\left(
%\begin{array}{c}
%d_{2}\\
%d_{1}\\
%d_{0}
%\end{array}
%\right).
%\end{align*}
%
%Let $A = \left(\begin{array}{ccc}
%-a&-a&-1\\
%1&0&0\\
%0&1&0
%\end{array}
%\right)
%$.
%Then 
%$$
%|xI - A| = x^3 + ax^2 + ax + 1 = (x +1)(x^2 + (a-1)x+1).
%$$
%
%Hence we have 
%\begin{enumerate}
%\item
%If $a = 3$, $x = -1$ is a triple solution.
%\item
%If $a = -1$, $x= -1$ or  $x = 1$ is a weight solution.
%\item
%If $a \not=-1, 3$, there are 3 different solutions.
%\end{enumerate}
%
%\vskip 2mm
%
%Let $f(x) = d_0 + d_1x + d_2x^2 +d_3x^3 + d_4s^4 + \cdots$.
%Then 
%\begin{align*}
%axf(x) &= ad_0x + ad_1x^2+ad_2x^3 + ad_3x^4+\cdots \\
%ax^2f(x) &= ad_0x^2 + ad_1x^3 +ad_2x^4 +\cdots\\
%x^3f(x)&= d_0x^3 + d_1x^4 +\cdots
%\end{align*}
%Hence 
%$
%(1 + ax + ax^2 + x^3)f(x) = 1$, that is, 
%$$
%f(x) = \dfrac{1}{1 + ax + ax^2 + x^3}.
%$$
%
Suppose  that $1 + ax + ax^2 + x^3 = 0$ has $3$ different solutions $\alpha, \beta, \gamma$.  
Then,
\begin{align*}
f(x) & = \frac{1}{(x-\alpha)(x-\beta)(x-\gamma)}  \\ 
&= \dfrac{1}{(\alpha-\beta)(\gamma - \alpha)}\dfrac{1}{\alpha - x} + \dfrac{1}{(\alpha - \beta)(\beta-\gamma)}\dfrac{1}{\beta-x}
+ \dfrac{1}{(\gamma - \alpha)(\beta - \gamma)}\dfrac{1}{\gamma - x}\\
&= \dfrac{1}{\alpha(\alpha-\beta)(\gamma - \alpha)}\dfrac{1}{1 - x/\alpha} + \dfrac{1}{\beta(\alpha - \beta)(\beta-\gamma)}\dfrac{1}{1-x/\beta}\\
& \qquad \qquad \qquad + \dfrac{1}{\gamma(\gamma - \alpha)(\beta - \gamma)}\dfrac{1}{1 - x/\gamma}\\
&= \sum_{k=1}^\infty \left(
\dfrac{1}{\alpha^{k+1}(\alpha-\beta)(\gamma - \alpha)} + \dfrac{1}{\beta^{k+1}(\alpha - \beta)(\beta-\gamma)}
+ \dfrac{1}{\gamma^{k+1}(\gamma - \alpha)(\beta - \gamma)}\right)x^k.
\end{align*}
Hence, for $1 \leq k $
$$
d_k = \dfrac{1}{\alpha^{k+1}(\alpha-\beta)(\gamma - \alpha)} + \dfrac{1}{\beta^{k+1}(\alpha - \beta)(\beta-\gamma)}
+ \dfrac{1}{\gamma^{k+1}(\gamma - \alpha)(\beta - \gamma)}.
$$
\end{proof}

%%%%%%%%%%%%%%%%%%%%%%%%%%%%%%%%%%%%%%%%%%%%%%%%
\section{Examples}\label{EXA}

\begin{exa}\label{exa:$(3,1)$}
Set $a = 3$ and $b=1$. Then we have 
\[  d_n = |D_n(3)| = \frac{1}{n!}\frac{d^n}{dx^n}(1+x)^{-3} |_{x=0} = \frac{(-1)^n}{2}(n+1)(n+2) \]
and
\[  |E_n^k(3)| = \frac{(-1)^{k-1}}{2}k(n+2)(n-k+2) .  \]
%\begin{align*}
%|D_n(3,1)| &= -3 |D_{n-1}(3,1)| - 3 |D_{n-2}(3, 1)| - |D_{n-3}(3, 1)| \ (n \geq 3), \\
%|D_1(3,1)| &= -3, |D_2(3,1)| = 6, |D_3(3,1)| = -10\\
%|F_n(3,1)| &= 3|F_{n-1}(3,1)|  - 3|F_{n-2}(3,1)|  + |F_{n-3}(3,1)| (n \geq 3), \\
%|F_1(3,1)| &= 3, |F_2(3,1)| = 6, |F_3(3,1)| = 10. 
%\end{align*}
%Then from the standard argument we have for $n \geq 1$ 
%
%\begin{align*}
%D_n(3,1) &= (-1)^{n-1}(-\dfrac{1}{2}n^2 - \dfrac{3}{2}n -1),\\
%F_n(3,1) &= \dfrac{1}{2}n^2 + \dfrac{3}{2}n + 1.
%\end{align*}
%
%Hence from Proposition~\ref{prp:$E_n$} we have for $1 \leq k \leq n+1$
%\begin{align*}
%|E_n^k(3,1)| = \dfrac{1}{2}nk^2 + k^2 - \dfrac{1}{2}n^2k - 2kn - 2k \ (k:\hbox{even}),\\
%|E_n^k(3,1)| = -\dfrac{1}{2}nk^2 - k^2 + \dfrac{1}{2}n^2k + 2kn + 2k \ (k:\hbox{odd}).\\
%\end{align*}
 Note that if $|E_n^k(3)| = 0$, then $k = n+2$. Therefore, for any $1 \leq k \leq n+1$ we have $|E_n^k(3)| \not=0,$ that is, all order-$n$ minors of $E_n(3)$ are nonzero.

\end{exa}

\vskip 2mm

\begin{exa}
Set $n = 10$ and $a=2$. Then $|E_{10}^k(2)| = 0$ for all $k\in \{1,2,\ldots,11\}$.

Indeed, the equation $x^3+2x^2+2x+1=(x+1)(x^2+x+1)=0$ has solutions $-1, \omega, \omega^2$, where $\omega=\dfrac{-1+\sqrt{3}\iota}{2}.$
Then we have
\begin{align*}
  d_k & =\frac{1}{(-1)^{k+1}(-1-\omega)(\omega^2+1)} + \frac{1}{\omega^{k+1}(-1-\omega)(\omega-\omega^2)}
              +\frac{1}{\omega^{2(k+1)}(\omega^2+1)(\omega-\omega^2)} \\
       & = (-1)^k - \frac{\omega^{k+2}(1-\omega^k)}{1-\omega}
 = \begin{cases} (-1)^k  & \text{ if } k\equiv 0 \text{ (mod 3)} \\  (-1)^k-1 & \text{ if } k\equiv 1 \text{ (mod 3)}  \\ (-1)^k +1  & \text{ if } k\equiv 2 \text{ (mod 3)} \end{cases}, 
\end{align*}

since $\omega^3=1$.
So the sequence $\{d_k\}_{k=-1}^\infty$ is periodic with the period 6 and 
\[  (d_0,d_1,d_2,d_3,d_4,d_5)=(1,-2,2,-1,0,0).  \]
By Theorem~\ref{thm:$E_n$} it is easily computed that $|E_{10}^1(2)|=|E_{10}^2(2)|=\cdots =|E_{10}^6(2)|=0$, and
we can get $|E_{10}^k(2)| = 0$ for all $k\in \{1,2,\ldots,11\}$.

%Indeed, from $x^3 + 2x^2 + 2x + 1 = 0$, it has three different solutions $\alpha = -1$, $\beta = \dfrac{-1+\sqrt{3}i}{2}$, $\gamma = \dfrac{-1-\sqrt{3}i}{2}$. 
%We note that $\alpha - \beta = \gamma$ and $\beta - \gamma = \sqrt{3}i$.
%
%Hence, for $1 \leq k \leq 11$
%\begin{align*}
%|E_{10}^5(2,1)| &= (-1)^{11-5}(d_{5-1}d_{10-5+1}-d_{5-2}d_{10-5})\\
%&= d_4d_6-d_3d_5\\
%d_k&= \dfrac{\sqrt{3}i+(-1)^k(\gamma^k-\beta^k)}{(-1)^k\gamma\sqrt{3}i}.
%\end{align*} 
%Since $d_4 = d_5=0$, we have $|E_{10}^5(2,1)|=d_4d_6-d_3d_5 = 0$.

\end{exa}

%%%%%%%%%%%%%%%%%%%%%%%%%%%%%%%%%%%%%%%%%%%%%%%%%%%%%
\section{Invertibility of \texorpdfstring{$E^k_n(a)$}{Ekn(a)}} \label{IoE}

In this section if $|a| \geq 5$, we show that for $1 \leq k \leq n+1$ the matrix $E_n^k(a)$, which is obtained by deleting the $k$-th row of $E_n(a)$, is invertible.

\vskip 3mm

The following lemma is useful.

\vskip 3mm

\begin{lem}\label{lem:down}
Let $a(i)$, $b(i)$, $c(i)$, $d(i) \in \C$  with $|a(i)| = |d(i)| = 1$ and $|b(i)| = |c(i)| = \alpha \geq 5$ $(i=1, 2, \dots, k-1)$. If

\begin{align*}
&c(1)x_1 + d(1)x_2=0\\
&b(2)x_1 + c(2)x_2 + d(2)x_3 = 0\\
&a(3)x_1 + b(3)x_2 + c(3)x_3 + d(3)x_4 = 0\\
&a(4)x_2 + b(4)x_3 + c(4)x_4 + d(4)x_5 = 0\\
&\cdots\\
&a(k-1)x_{k-3} +b(k-1)x_{k-2} + c(k-1)x_{k-1} + d(k-1)x_k = 0,
\end{align*}
then $|x_{l+1}| \geq (\alpha -2 )|x_l|$ \ $(l = 1, 2, \dots, k-1)$.
\end{lem}

\begin{proof}
Since $c(1)x_1 + d(1)x_2 = 0$, we have $|x_2| = \alpha|x_1| \geq (\alpha-2)|x_1|$. Since $b(2)x_1 + c(2)x_2 + d(2)x_3 = 0$, we have 
$$
|x_3| = |-c(2)x_2 - b(2)x_1| \geq \alpha|x_2| - \alpha|x_1| \geq (\alpha -1)|x_2| \geq (\alpha - 2)|x_2|.
$$
We assume that $|x_{l+1}| \geq (\alpha - 2)|x_l|$ and $|x_{l+2}| \geq (\alpha - 2)|x_{l+1}|$. Then the relation $a(l+2)x_l + b(l+2)x_{l+1}+c(l+2)x_{l+2} + d(l+2)x_{l+3}=0$ implies that
\begin{align*}
|x_{l+3}| &\geq |c(l+2)x_{l+2}| - |b(l+2)x_{l+1}| - |a(l+2)x_l|\\
&= \alpha|x_{l+2}| - \alpha|x_{l+1}| -|x_l|\\
&\geq \alpha|x_{l+2}| - \frac{\alpha}{\alpha-2}|x_{l+2}| - \frac{1}{(\alpha-2)^2}|x_{l+2}|\\
&\geq \alpha|x_{l+2}| - \frac{\alpha}{\alpha-2}|x_{l+2}| - \frac{1}{\alpha-2}|x_{l+2}|\\
&= \left(\alpha - \frac{\alpha + 1}{\alpha-2}\right)|x_{l+2}| \geq (\alpha - 2)|x_{l+2}|.
\end{align*} 
Hence, we have $|x_{l+1}| \geq (\alpha -2 )|x_l|$ \ $(l = 1, 2, \dots, k-1)$ inductively.
\end{proof}

\vskip 3mm

\begin{prp}\label{prp:$a>5$}
For $a\in \C$ with $|a|\ge 5$ and $1\le k \le n+1$ the matrix $E_n^k(a)$ is invertible, that is, all order-$n$ minors of $E_n(a)$ are nonzero.
\end{prp}

\begin{proof}
For $1\le k \le n+1$ and $\boldsymbol{x} =(x_1,x_2,\ldots ,x_n)^t  \in \C^n$, we will show that if 
$E_n^k(a)\boldsymbol{x}=\boldsymbol{0}$, then $\boldsymbol{x}= \boldsymbol{0}$.

(1) When $k=n+1$, $E_n^{n+1}(a)\boldsymbol{x} =\boldsymbol{0}$ means that
\begin{gather*}
  x_{-1}=x_0=0, \\
  -x_i + ax_{i+1}-ax_{i+2}+x_{i+3} =0 \; (i=-1,0,\ldots,n-3)   \tag{*} \\
  -x_{n-2} + a x_{n-1} - a x_n = 0  \tag{**}
\end{gather*} 
Apply Lemma~\ref{lem:down} for (*), we have
\[   |x_{i+1}| \ge (|a|-2) |x_i|, \quad i=1,2,\ldots, n-1.  \]
By (**) we have
\begin{align*}
  0 & = |-x_{n-2} + a x_{n-1} - a x_n | \ge |a| |x_n| - |a| |x_{n-1}| - |x_{n-2}|  \\
     & \ge ( |a| - \frac{|a|}{|a|-2} - \frac{1}{(|a|-2)^2} ) |x_n| \ge |x_n| .
\end{align*}
So we can get $0=x_n =x_{n-1}= \cdots =x_1$.

(2) When $k=1$, we set $\boldsymbol{y}=(x_n, x_{n-1},\ldots, x_1)^t\in \C^n$.
The equation $E_n^1(a)\boldsymbol{x}=\boldsymbol{0}$ is equivalent to 
$E_n^{n+1}(a) \boldsymbol{y}=\boldsymbol{0}$.
So we have $\boldsymbol{y}=\boldsymbol{0}$, that is $\boldsymbol{x}=\boldsymbol{0}$.

(3) When $2\le k \le n$, $E_n^k(a)\boldsymbol{x} = \boldsymbol{0}$ means that
\begin{gather*}
  x_{-1}=x_0=0, \;  x_{n+1}=x_{n+2}=0  \\
  -x_i + a x_{i+1} - a x_{i+2} + x_{i+3} = 0 \quad (i=-1,0,\ldots, k-3)  \tag{***} \\
  -x_i + a x_{i+1} - a x_{i+2} + x_{i+3} = 0 \quad (i=k-1,k,\ldots, n-1) , \tag{****} 
\end{gather*}
Apply Lemma~\ref{lem:down} for (***) and (****) like as (1) and (2), we have 

\begin{gather*}
  |x_{i+1}|\ge (|a|-2) |x_i|,  \quad  i=1,2,\ldots, k-1 \\ 
\intertext{and}
  |a_{i-1}| \ge (|a|-2) |x_i|, \quad  i=k,k+1, \ldots, n  .
\end{gather*}
Since $|x_k|\ge (|a|-2)|x_k|$ and $|x_{k-1}|\ge (|a|-2)|x_k|$, we have $x_{k-1}=x_k=0$
and $\boldsymbol{x}=\boldsymbol{0}$.
\end{proof}

\vskip 3mm

\begin{rmk}\label{rmk:$a > 5$}
Proposition~\ref{prp:$a>5$} is true even if we replace $E_n(a)$ by $\tilde{E}_n(a)$  with $|a| \ge 5$, where 
$$
\tilde{E}_n(a)
= \sum_{j=2}^n|j-1\rangle \langle j| + a \sum_{j=1}^n |j\rangle \langle j| + a \sum_{j=1}^n |j+1\rangle \langle j| + \sum_{j=1}^{n-1} |j+2\rangle \langle j| .
%\left(
%\begin{array}{ccccccc}
%a&1&0&0&0&\cdots&0\\
%a&a&1&0&0&\cdots&0\\
%1&a&a&1&0&\cdots&0\\
%0&1&a&a&1&\ddots&\vdots\\
%\vdots&\ddots&\ddots&\ddots&\ddots&\ddots&0\\
%\vdots&&\ddots&\ddots&\ddots&\ddots&1\\
%0&\cdots&\cdots&0&1&a&a\\
%0&\cdots&\cdots&\cdots&0&1&a
%\end{array}
%\right).
$$
\hfill$\qed$
\end{rmk}

\vskip 3mm

\begin{lem}\label{lem:$H_n(a)$}
Let $a$ be a complex number with $|a|\ge 2$ and set
\[  H_n(a) = a \sum_{i=1}^n |i\rangle \langle i| + \sum_{i=1}^{n-1}
      (|i\rangle \langle i+1| + |i+1\rangle \langle i|) \in \mathbb{M}_n (\mathbb{C}) . \]
Then $H_n(a)$ is invertible.
\end{lem}

\begin{proof}
We set $e_n =|H_n(a)|$.
Then the sequence $\{e_n\}_{n=1}^\infty$ satisfies
\[  e_o=1, \; e_1 =a, \; e_{n+2} = a e_{n+1} -e_n \quad (n=0,1,2,\ldots) . \]
Let $x^2-ax+1=(x-\alpha)(x-\beta)$. 
Then we can get
\[   e_n = \frac{\alpha^{n+1}-\beta^{n+1}}{\alpha -\beta} \qquad  n=0,1,2,\ldots . \]
If $\alpha^{n+1}-\beta^{n+1}=0$ for some $n$, then we have $\alpha=\beta=1$ or
$\alpha =\beta=-1$ because $|\alpha + \beta|=|a|\ge 2$ and $\alpha \beta=1$.
This implies, for $n=0,1,2,\ldots$,
\[    e_n = \begin{cases}  n+1 \; &  \text{ if } \alpha=\beta=1  \\
                                         (-1)^n(n+1)   & \text{ if } \alpha = \beta = -1 \end{cases} . \]
So we have $e_n \neq 0$, that is, $H_n(a)$ is invertible for $|a|\ge 2$.
\end{proof}

\vskip 3mm

For any $n \in \N$ and a number $a \in \C$, set $(n +2) \times n$ matrix 

$$
G_n(a)= \sum_{j=1}^n (|j\rangle + a |j+1\rangle + |j+2 \rangle) \langle j|  \in \M_{n+2,n}(\C).  
%\left(
%\begin{array}{ccccccc}
%1&0&0&0&0&\cdots&0\\
%a&1&0&0&0&\cdots&0\\
%1&a&1&0&0&\cdots&0\\
%0&1&a&1&0&\ddots&\vdots\\
%\vdots&\ddots&\ddots&\ddots&\ddots&\ddots&0\\
%\vdots&&\ddots&\ddots&\ddots&\ddots&1\\
%0&\cdots&\cdots&0&0&1&a\\
%0&\cdots&\cdots&\cdots&0&0&1
%\end{array}
%\right).
$$

From the same argument as in Proposition~\ref{prp:$a>5$} we can get the following:

\vskip3mm

\begin{prp}\label{prp:BD}
 For $1\le i <j \le n+2$ and a number $a\in \mathbb{C}$ with $|a|\ge 2$,
let $G_n^{i,j}(a)$ be a matrix which is obtained by deleting $i$-th and $j$-th rows of $G_n(a)$.
Then $G_n^{i,j}(a)$ is invertible, that is, all order-$n$ minors of $G_n(a)$ are nonzero.
\end{prp}

\begin{proof}
For $1\le i<j\le n+2$ and $\boldsymbol{x} =(x_1,x_2,\ldots,x_n)^t \in \mathbb{C}^n$,
we will show that if $G_n^{i,j}(a)\boldsymbol{x}=\boldsymbol{0}$, then 
$\boldsymbol{x}=\boldsymbol{0}$.

When $i=1, j=n+2$, $G_n^{1,n+2}(a)\boldsymbol{x}=H_n(a)\boldsymbol{x}=\boldsymbol{0}$ 
implies $\boldsymbol{x}=\boldsymbol{0}$ by Lemma~\ref{lem:$H_n(a)$}.

When $1<i<j=n+2$, $G_n^{i,n+2}(a)\boldsymbol{x}=\boldsymbol{0}$ is equivalent to
\[  J_{i-1}(a) (x_1,x_2, \ldots,x_{i-1})^t = \boldsymbol{0} \text{ and }
    H_{n-i+2}^1(a) (x_{i-1},x_i,\ldots, x_n)^t = \boldsymbol{0},  \]
where $H^1_{n-i+2}$ is obtained by deleting the 1st row of $H_{n-i+2}$.
\[  J_k(a) = \sum_{i=1}^{k-2} (|i\rangle  + a|i+1\rangle  +|i+2\rangle) \langle i| 
                 + (|k-1\rangle + a|k\rangle ) \langle k-1| + |k\rangle \langle k| .  \]
Since $J_k(a)$ is invertible, this condition is also equivalent to
\[  x_1= x_2=\cdots =x_{i-1}=0  \text{ and }
    H_{n-i+1}(a) (x_i, x_{i+1}, \ldots, x_n)^t = \boldsymbol{0}.  \]
By Lemma~\ref{lem:$H_n(a)$}, we have $\boldsymbol{x}=\boldsymbol{0}$.

When $1=i<j < n+2$,  $G_n^{1,j}(a) \boldsymbol{x} = \boldsymbol{0}$ is equivalent to
\[   G_n^{n-j+3, n+2}(a) (x_n, x_{n-1}, \ldots, x_1)^t = \boldsymbol{0}.  \]
So we have $\boldsymbol{x}=\boldsymbol{0}$.

When $1<i<j<n+2$, $G_n^{i,j}(a) \boldsymbol{x}=\boldsymbol{0}$ is equivalent to
\begin{gather*}
   J_{i-1}(a)(x_1,x_2,\ldots, x_{i-1})^t = \boldsymbol{0},  \\
   x_k + a x_{k+1} + x_{k+2} = 0 \quad (k=i-1,i,\ldots, j-1),  \\
   \text{and } J_{n-j+2}(a) (x_n,x_{n-1},\ldots,x_{j-1})^t = \boldsymbol{0},
\end{gather*}
and also equivalent to
\begin{gather*}
  x_1=x_2=\cdots =x_{i-1}= 0, \; x_{j-1}=x_j=\cdots = x_n=0,  \\
  \text{and } H_{j-i-1}(a) (x_i, x_{i+1},\ldots, x_{j-2})^t = \boldsymbol{0} .  
\end{gather*}
So we have $\boldsymbol{x}=\boldsymbol{0}$.
\end{proof}

%%%%%%%%%%%%%%%%%%%%%%%%%%%%%%%%%%%%%%%%%%%%%%%%%%%%%
\section{Subspaces of Maximal Dimension with Bounded Schmidt Rank} \label{SMDBSR}
In this section we  use the $(n + 1) \times n$ matrix $E_n(a)$ $($$a = 3$ or $|a| \geq 5$$)$ to the construction of subspaces of $\C^m \otimes \C^n$ of maximal dimension, which do not contain any nonzero vector of Schmidt rank less than 4. Moreover, using the $(n+2) \times n$ matrix $G_n(a)$ we can construct subspaces $\mathcal{T} \subset \mathcal{S}$ of $\C^m \otimes \C^n$ such that any vector in $\mathcal{T}\backslash\{\boldsymbol{0}\}$ has Schmidt rank greater  than or equal 4, but there is at least one vector with Schmidt rank $3$ in $\mathcal{S}\backslash{\mathcal{T}}$. 
\vskip 3mm

\vskip 3mm

For $n\in \mathbb{N}$ and a number $a\in \mathbb{C}$, we consider the $(n+3)\times n$
matrix $B_n(a)$ (resp. $\tilde{B}_n(a)$) as follows:
\[  B_n(a) = \sum_{i=1}^n ( |i\rangle -a |i+1 \rangle + a |i+2 \rangle -|i+3\rangle)\langle i| . \]

\[ (\hbox{resp.}\ \tilde{B}_n(a) =  \sum_{i=1}^n ( |i\rangle + a |i+1 \rangle + a |i+2 \rangle + |i+3\rangle)\langle i|.)
\]  

%For $n \in \N$ and a number $a \in \R$, we consider the $(n+3) \times n$ matrix $\textcolor{red}{B_n(a)}$ (resp. $\textcolor{red}{\tilde{B}_n(a)}$) as follows:
%$$ 
%\textcolor{red}{B_n(a)} = 
%\left(
%\begin{array}{c}
%b_1\\
%\textcolor{red}{
%E_n(a)}\\
%b_{n+3}
%\end{array}
%\right) 
%\ (\mbox{resp. } 
%\textcolor{red}{\tilde{B}_n(a)} =
%\left(
%\begin{array}{c}
%b_1\\
%\textcolor{red}{\tilde{E}_n(a)}\\
%\tilde{b}_{n+3}
%\end{array}
%\right) ),
%$$
%where $b_1 = (1,0,\dots,0)$, $b_{n+3} = (0,\dots,0,-1)$ and $\tilde{b}_{n+3}=(0,\dots,0,1)$.

\vskip 3mm

\begin{prp}\label{prp:invertible}
For $1\le i <j <k \le n+3$ and $|a|\ge 5$, let $B_n^{i,j,k}$ be a matrix which is obtained by
deleting $i$-th, $j$-th and $k$-th rows of $B_n(a)$.
Then $B_n^{i,j,k}(a) \in \mathbb{M}_n(\mathbb{C})$
is invertible.

\end{prp}

\vskip 3mm

\begin{proof}
It suffices to show that $B_n^{i,j,k}(a) \boldsymbol{x} = \boldsymbol{0}$ implies 
$\boldsymbol{x} = \boldsymbol{0}$, where $\boldsymbol{x} =(x_1,x_2,\ldots,x_n)^t$.

(1)  When $1=i<j<k =n+3$, $B_n^{1,j,n+3}(a) = E_n^{j-1}(a)$ is invertible by Proposition~\ref{prp:$a>5$}.
So $\boldsymbol{x}=\boldsymbol{0}$.

(2) When $1<i$, $i+1<j<k=n+3$, $B_n^{i,j,n+3}(a) \boldsymbol{x}=\boldsymbol{0}$ is
equivalent to 
\begin{gather*}
  x_{-2}=x_{-1}=x_0=0, \; x_{n+1}=x_{n+2}=0  \\
  \text{and } -x_{l-3}+ax_{l-2}-ax_{l-1}+x_l =0, 
\end{gather*}
where  $l \in \{1,\ldots,i-1\} \cup \{ i+1,\ldots,j-1\} \cup \{j+1,\ldots,n+2\}$.
This condition is also equivalent to
\begin{gather*}
  x_{-2}=x_{-1}=x_0=x_1=\cdots =x_{i-1}=0, \; x_{n+1}=x_{n+2}=0  \\
  \text{and } -x_{l-3}+ax_{l-2}-ax_{l-1}+x_l =0, 
\end{gather*}
where  $l \in \{ i+1,\ldots,j-1\} \cup \{j+1,\ldots,n+2\}$.
Then we have, by Lemma~\ref{lem:down},
\begin{align*}
  & |x_{l+1}| \ge (|a|-2)|x_l|, \quad l \in \{i, \ldots, j-2\}  \\
 \text{and } & |x_l| \ge (|a|-2) |x_{l+1}|, \quad l \in \{j-2,\ldots, n-1\}.
\end{align*}
Since $|x_{j-1}|\ge (|a|-2)|x_{j-2}|$ and $|x_{j-2}|\ge (|a|-2)|x_{j-1}|$,
this implies $x_{j-1}=x_{j-2}=0$ and $\boldsymbol{x}=\boldsymbol{0}$.

(3) When $1<i$, $j=i+1<k=n+3$, $B_n^{i,i+1,n+3}(a) \boldsymbol{x}=\boldsymbol{0}$ is
equivalent to 
\begin{gather*}
  x_{-2}=x_{-1}=x_0=0, \; x_{n+1}=x_{n+2}=0  \\
  \text{and } -x_{l-3}+ax_{l-2}-ax_{l-1}+x_l =0, 
\end{gather*}
where  $l \in \{1,\ldots,i-1\} \cup \{ i+2,\ldots,n+2\}$.
This condition is also equivalent to
\begin{gather*}
  x_{-2}=x_{-1}=x_0=x_1=\cdots =x_{i-1}=0, \; x_{n+1}=x_{n+2}=0  \\
  \text{and }  E_{n-i+1}^1(a) (x_i,x_{i+1}, \ldots, x_n)^t = \boldsymbol{0}. 
\end{gather*}
By Proposition~\ref{prp:$a>5$}, we have $x_i=x_{i+1}=\cdots =x_n=0$  and $\boldsymbol{x}=\boldsymbol{0}$.

(4) When $1<i<j=n+2, k=n+3$, $B_n^{i,n+2,n+3}(a) \boldsymbol{x}=\boldsymbol{0}$ is
equivalent to 
\begin{gather*}
  x_{-2}=x_{-1}=x_0=0, \; x_{n+1}=0  \\
  \text{and } -x_{l-3}+ax_{l-2}-ax_{l-1}+x_l =0, 
\end{gather*}
where  $l \in \{1,\ldots,i-1\} \cup \{ i+1,\ldots,n+1\}$.
This condition is also equivalent to
\begin{gather*}
  x_{-2}=x_{-1}=x_0=x_1=\cdots =x_{i-1}=0, \; x_{n+1}=0  \\
  \text{and }  E_{n-i+1}^{n-i+2}(a) (x_i,x_{i+1}, \ldots, x_n)^t = \boldsymbol{0}. 
\end{gather*}
By Proposition~\ref{prp:$a>5$}, we have $x_i=x_{i+1}=\cdots =x_n=0$  and $\boldsymbol{x}=\boldsymbol{0}$.

(5) When $i=1<j<k<n+3$, $B_n^{1,j,k}(a) \boldsymbol{x}=\boldsymbol{0}$ is
equivalent to 
\[  B_n^{n-k+4,n-j+4,n+3}(a) (x_n,x_{n-1},\ldots,x_1)^t = \boldsymbol{0}. \]
So $\boldsymbol{x}=\boldsymbol{0}$ by (2), (3) and (4).

(6) When $1<i, j=i+1, k=i+2<n+3$,  $B_n^{i,i+1,i+2}(a) \boldsymbol{x}=\boldsymbol{0}$ is
equivalent to 
\begin{gather*}
  x_{-2}=x_{-1}=x_0=0, \; x_{n+1}=x_{n+2}=x_{n+3}=0  \\
  \text{and } -x_{l-3}+ax_{l-2}-ax_{l-1}+x_l =0, 
\end{gather*}
where  $l \in \{1,\ldots,i-1\} \cup \{ i+3,\ldots,n+3 \}$.
So we have $x_1=x_2=\cdots=x_{i-1}=0$ and $x_i=x_{i+1}=\cdots=x_n=0$, that is,
$\boldsymbol{x}=\boldsymbol{0}$.

(7) When $1<i, i+1<j, k=j+1<n+3$, $B_n^{i,j,j+1}(a) \boldsymbol{x}=\boldsymbol{0}$ is
equivalent to 
\begin{gather*}
  x_{-2}=x_{-1}=x_0=0, \; x_{n+1}=x_{n+2}=x_{n+3}=0  \\
  \text{and } -x_{l-3}+ax_{l-2}-ax_{l-1}+x_l =0, 
\end{gather*}
where  $l \in \{1,\ldots,i-1\} \cup \{ i+1,\ldots,j-1\} \cup \{j+2,\ldots, n+3\}$.
This condition is also equivalent to
\begin{gather*}
  x_{-2}=x_{-1}=x_0=x_1=\cdots =x_{i-1}=0, \\
  x_{j-1}=x_j=\cdots=x_n=x_{n+1}=x_{n+2}=x_{n+3}=0  \\
  \text{and }  E_{j-i-1}^{j-i}(a) (x_i,x_{i+1}, \ldots, x_{j-2})^t = \boldsymbol{0}. 
\end{gather*}
By Proposition~\ref{prp:$a>5$}, we have $x_i=x_{i+1}=\cdots =x_{j-2}=0$  and $\boldsymbol{x}=\boldsymbol{0}$.

(8) When $1<i, j=i+1, j+1<k<n+3$, $B_n^{i,i+1,k}(a) \boldsymbol{x}=\boldsymbol{0}$ is
equivalent to 
\[  B_n^{n-k+4,n-i+3,n-i+4}(a) (x_n,x_{n-1},\ldots,x_1)^t = \boldsymbol{0}. \]
So $\boldsymbol{x}=\boldsymbol{0}$ by (7).

(9) When $1<i$, $i+1<j$, $j+1<k<n+3$, $B_n^{i,j,k}(a) \boldsymbol{x}=\boldsymbol{0}$ is
equivalent to 
\begin{gather*}
  x_{-2}=x_{-1}=x_0=0, \; x_{n+1}=x_{n+2}=x_{n+3}=0  \\
  \text{and } -x_{l-3}+ax_{l-2}-ax_{l-1}+x_l =0, 
\end{gather*}
where  $l \in \{1,\ldots,i-1\} \cup \{ i+1,\ldots,j-1\} \cup \{j+1,\ldots,n+3\}$.
This condition implies 
\begin{gather*}
   x_1=\cdots =x_{i-1}=0, \; x_{k+1}=\cdots =x_{n}=0, \\
   |x_l| \ge (|a|-2)|x_{l-1}| \quad l \in \{i+1, \ldots, j-1\}  \\
  \text{and }  |x_l| \ge (|a|-2) |x_{l+1}|, \quad l \in \{j-2,\ldots, k-3\}. 
\end{gather*}
Since $|x_{j-1}|\ge (|a|-2)|x_{j-2}|$ and $|x_{j-2}|\ge (|a|-2)|x_{j-1}|$,
we have $x_{j-1}=x_{j-2}=0$ and $\boldsymbol{x}=\boldsymbol{0}$.
\end{proof}

\vskip 3mm

\begin{rmk}\label{rmk:$a=3$}
We consider the following system of equations:
\begin{align*}
  -3x_1+x_2 & = 0 \\
  3x_1 -3x_2 +x_3 & = 0 \\
 -x_{l-3}+3x_{l-2}-3x_{l-1}+x_l & =0 \quad (l=4,5,\ldots,n).
\end{align*}
Then we have
\[  x_n = \frac{n(n+1)}{2}x_1 .  \]
This means that, if $x_1\neq 0$, then $|x_1|<|x_2|<\cdots < |x_n|$.
Applying this fact for (1) and (9) in the proof of Proposition~\ref{prp:invertible}, we can show that
$B_n^{i,j,k}(3)$ is invertible for $1\le i <j <k \le n+3$.
\end{rmk}

\vskip 3mm

\begin{prp}\label{prp:invertible_tilde}
For $1 \leq i < j < k \leq n+3$ and $|a| \geq 5$ let $\tilde{B}_n^{i,j,k}(a)$ be a matrix which is obtained by deleting the $i, j, k$-th rows of $\tilde{B}_n(a)$.
Then $|\tilde{B}_n^{i,j,k}(a)| \not=0$.
\end{prp}

\begin{proof}
It follows from the same argument as in Proposition~\ref{prp:invertible} using Remark~\ref{rmk:$a > 5$}.
\end{proof}

\vskip 3mm

\begin{cor}\label{cor:invertible}
For $a = 3$ or $|a| \geq 5$, the columns of $B_n(a)$ are linearly independent such that any nonzero linear combination of these columns has at least $4$ nonzero entries.
\end{cor}

\begin{proof}
Suppose that there are $\lambda_1, \lambda_2, \dots, \lambda_n \in \C$ such that $\sum\limits_{i=1}^n\lambda_ie_i = (x_1, x_2, \dots, x_{n+3})^t$ has less than $4$ nonzero entries, where $e_i$ \ $(1\leq i \leq n)$ are columns in $B_{n}(a)$ and $x_i \in \C$ \ $(1 \leq i \leq n+3)$. Assume that for distinct elements $l,j$ and $k$ in $\{1,2,\ldots, n+3 \},$ we have $x_i = 0$ for all $i \in \{1,2,\ldots, n+3 \}\backslash \{l, j, k\}.$ Then
we have 
$$
B_n^{l,j,k}(a)
\left(\begin{array}{c}
\lambda_1\\
\vdots \\
\lambda_n
\end{array}
\right)
=
\left(\begin{array}{c}
0\\
\vdots\\
0
\end{array}
\right).
$$
Since $B_n^{i,j,k}(a)$ is invertible by Proposition~\ref{prp:invertible} and Remark~\ref{rmk:$a=3$}, we get 
each $\lambda_i = 0$ $(1 \leq i \leq n)$, that is, $\sum\limits_{i=1}^n \lambda_ie_i = (x_1, x_2, \dots, x_{n+3})^t = \boldsymbol{0}$. This is a contradiction.
\end{proof}

\vskip 3mm

\begin{cor}\label{cor:$a_geq_5$}
For $a=3$ or $|a| \geq 5$ the columns of $\tilde{B}_n(a)$ are linearly independent such that any nonzero linear combination of these columns has at least $4$ nonzero entries.
\end{cor}

\begin{proof}
When $|a| \geq 5$, it follows from Proposition~\ref{prp:invertible_tilde} and the same argument as in Corollary~\ref{cor:invertible}.

When $a=3$, the matrix $\tilde{B}_{n}(3)$ implies the following system of equations:
\begin{align*}
  3x_1+x_2 & = 0 \\
  3x_1 +3x_2 +x_3 & = 0 \\
 x_{l-3}+3x_{l-2}+3x_{l-1}+x_l & =0 \quad (l=4,5,\ldots,n).
\end{align*}
Then we have
\[  x_n = (-1)^{n}\frac{n(n+1)}{2}x_1   \]
and $|x_{1}| <|x_{2}| < \cdots <|x_{n}|$ if $x_{1}\neq 0$.
By the similar reason as Remark~\ref{rmk:$a=3$}, we can show  that
$\tilde{B}_n^{i,j,k}(3)$ is invertible for $1\le i <j <k \le n+3$.
\end{proof}

\vskip 3mm

\begin{thm}\label{thm:subspace}
Let $m$ and $n$ be natural numbers such that $4 \leq \min\{m,n\}$. 
Let $N = n + m -2$, and $\{|e_i\rangle\}_{i=1}^{m}$ (resp. $\{|f_j\rangle\}_{j=1}^{n}$) be the canonical basis for $\C^m$ (resp. $\C^n$). For $3 \leq d \leq N -3$ define 
\begin{align*}
\mathcal{S}^{(d)} = \mathrm{span}\{|e_{i-1}\rangle \otimes |f_{j+2}\rangle &-a|e_{i}\rangle \otimes |f_{j+1}\rangle + a|e_{i+1}\rangle \otimes |f_{j}\rangle - |e_{i+2}\rangle \otimes |f_{j-1}\rangle \colon \\
&2 \leq i \leq m-2, 2 \leq j \leq n-2, i + j = d +1\},\\
\mathcal{S}^{(0)} = \mathcal{S}^{(1)} = \mathcal{S}^{(2)} &= \mathcal{S}^{(N-2)} = \mathcal{S}^{(N-1)} = \mathcal{S}^{(N)} = \{0\}
\end{align*}
and $\mathcal{S} =  \bigoplus\limits_{d=0}^N\mathcal{S}^{(d)}$. 
If $a = 3$ or $|a| \geq 5$, then $\mathcal{S}$ does  not contain any nonzero vector of Schmidt rank $\leq 3$ and $\dim \mathcal{S} = (m - 3)(n - 3)$. 
\end{thm}

\begin{proof}
Let $\phi\colon \C^m \otimes \C^n \rightarrow M_{m\times n}(\C)$ be the linear isomorphism defined by $\phi(|\eta\rangle) = [c_{ij}]$ for each $|\eta\rangle = \sum_{i,j}c_{ij}|e_i\rangle \otimes |f_j\rangle  \in \C^m \otimes \C^n$.  Then $|\eta\rangle$ has Schmidt rank at least $r$ if and only if the corresponding matrix $[c_{ij}]$ is of rank at least $r$. 
Also, it is known that a matrix has rank
at least $r$ if and only if it has a nonzero minor of order $r$. Thus,
it is enough to construct a set
of $(m-3)(n-3)$ linearly independent matrices, which are image under $\phi$ of a basis of $\mathcal{S},$ such that
any nonzero linear combination of these
matrices has a nonzero minor of order 4.

Label the anti-diagonals of any $m \times n$ matrix by non-negative integers $k,$ such that the first
anti-diagonal from upper left (of length one) is labeled $k=0$  and value of $k$ increases from upper
left to lower right. Let the length of the $k$-th anti-diagonal be denoted by

 $[k]$. 
Note that if $m \leq n$, then
$$
[k] = \left\{\begin{array}{ccc}
k+1&\hbox{for}&k \in \{3, 4, \dots, m-2\}\\
% |m-1| = 
m&\hbox{for}&k \in \{m-1, m, \dots, n-1\}\\
n + m - k - 1&\hbox{for} & k \in \{n, n+1, \dots, N-3 = n + m - 5\}
\end{array}
\right.
$$

 Recall that for $3 \leq k \leq N-3$, $\mathcal{S}^{(k)}$ is generated by the set
\begin{align*}
 B_k=\{|e_{i-1}\rangle \otimes |f_{j+2}\rangle &-a|e_{i}\rangle \otimes |f_{j+1}\rangle + a|e_{i+1}\rangle \otimes |f_{j}\rangle - |e_{i+2}\rangle \otimes |f_{j-1}\rangle \colon \\
&2 \leq i \leq m-2, 2 \leq j \leq n-2, i + j = k + 1\}.
\end{align*}

 Note that any
element of $\phi(B_k)$ is the matrix obtained from the $m \times n$ zero matrix by
replacing the $k$-th anti-diagonal
by $(0, \dots, 0, -1, a, -a, 1, 0, \dots, 0).$ 
For  $3 \leq k \leq N-3$, from Corollary~\ref{cor:invertible}, $\phi(B_k)$
is a set of $[k]-3$ linearly independent matrices.  Also, by Corollary~\ref{cor:invertible} it follows that if $M$
is a matrix obtained by taking any nonzero linear combination of
matrices from $\phi(B_k),$ then $M$ has at least $4$ nonzero entries in the $k$-th
anti-diagonal and the entries of $M$,  other than  those on $k$-th anti-diagonal, are zeros. Since the
determinant of the $4 \times 4$ submatrix with these $4$ nonzero elements in its principal
anti-diagonal is clearly nonzero, any linear combination of the matrices in $\phi(B_k)$ has at least
one nonzero order-$4$ minor, thus has rank at least $4.$

Let $B= \displaystyle \bigcup_{k=3}^{N-3} B_k$. Since elements from different 
$\phi(B_k)$ have different
nonzero anti-diagonal, $\phi(B)$ is linearly independent. Thus $B$
is a basis for $\mathcal{S}.$ Let $C$ be a matrix obtained by an arbitrary
linear combination from the elements of $\phi(B)$. Let $\kappa$ be the largest $k$ for which the linear
combination involves an element from $\phi(B_k)$. The $\kappa$-th anti-diagonal of $C$ has at least
$4$ nonzero elements. Because $\kappa$ labels the bottom-rightmost anti-diagonal of $C$ that contains nonzero
elements, the $4 \times 4$ submatrix of $C,$ with these 4 nonzero elements in
the principal anti-diagonal, has only zero entries in all its anti-diagonals which are below the principal
anti-diagonal. Hence  the $4 \times 4$ submatrix has nonzero determinant. Thus, the rank of $C$ is at least 4. We
conclude that  $\mathcal{S}$  does not contain any nonzero vector of Schmidt rank $\leq 3.$

The dimension of $\mathcal{S}$ is equal to the cardinality of $B$. 
Hence, the dimension of $\mathcal{S}$ is given by

\begin{align*}
 %|B| &= \sum_{k=3}^{N-3} |B_k|\\
\sum_{k=3}^{N-3} ([k]-3)
% &= 
 &=\sum_{k=3}^{m-2}(k-2) + \sum_{k=m-1}^{n-1}(m - 3) + \sum_{k=n}^{n+m-5}(n + m - k - 4)\\
 &= 2\sum_{k=3}^{m-2}(k - 2) + (m -3)(n - m +1)\\
 &= (m - 3)(n - 3).
\end{align*} 
A similar argument holds in the case $m > n$.
%Thus, $|B|$ is same as the number of entries in an $(m-3) \times (n-3)$ matrix, i.e., $|B|=  (m-3)(n-3).$ 
\end{proof}

\vskip 3mm

\begin{rmk}
From the general fact on Schmidt rank it follows that 
% for $a=3$  or $\textcolor{red}{|a| \geq 5}$, 
all the elements of the basis 
\begin{align*}
 B=\bigcup_{d=3}^{N-3} \{|e_{i-1}\rangle \otimes |f_{j+2}\rangle &-a|e_{i}\rangle \otimes |f_{j+1}\rangle + a|e_{i+1}\rangle \otimes |f_{j}\rangle - |e_{i+2}\rangle \otimes |f_{j-1}\rangle \colon \\
&2 \leq i \leq m-2, 2 \leq j \leq n-2, i + j = d + 1\},
\end{align*}
of $\mathcal{S}$ have Schmidt rank $4$ if and only if $a\neq 0$.
\end{rmk}

\vskip 3mm

\begin{rmk}\label{rmk:subspace}
Theorem~\ref{thm:subspace} is true even if we replace $\mathcal{S}$ by $\tilde{\mathcal{S}}$ with $a=3$ or $|a| \geq 5$ as follows:

\begin{align*}
\mathcal{\tilde{S}}^{(d)} = \mathrm{span}\{|e_{i-1}\rangle \otimes |f_{j+2}\rangle &+a|e_{i}\rangle \otimes |f_{j+1}\rangle + a|e_{i+1}\rangle \otimes |f_{j}\rangle + |e_{i+2}\rangle \otimes |f_{j-1}\rangle \colon \\
&2 \leq i \leq m-2, 2 \leq j \leq n-2, i + j = d + 1\},\\
\mathcal{\tilde{S}}^{(0)} = \mathcal{\tilde{S}}^{(1)} = \mathcal{\tilde{S}}^{(2)} &= \mathcal{\tilde{S}}^{(N-2)} = \mathcal{\tilde{S}}^{(N-1)} = \mathcal{\tilde{S}}^{(N)} = \{0\}
\end{align*}
and $\mathcal{\tilde{S}} =  \bigoplus\limits_{d=0}^N\mathcal{\tilde{S}}^{(d)}$. 
Note that it follows from Corollary~\ref{cor:$a_geq_5$}.
\end{rmk}

\vskip 3mm

\begin{thm}\label{thm:subspaces}
Let $m$ and $n$ be natural numbers such that $4 \leq \min\{m,n\}$. 
Let $N = n + m -2$, and $\{|e_i\rangle\}_{i=1}^{m}$ (resp. $\{|f_j\rangle\}_{j=1}^{n}$) be the canonical basis for $\C^m$ (resp. $\C^n$). For $3 \leq d \leq N -3$ define 
\begin{align*}
\mathcal{T}^{(d)} = \mathrm{span}\{|e_{i-1}\rangle \otimes |f_{j+2}\rangle &+(a + 1)|e_{i}\rangle \otimes |f_j\rangle + (a + 1)|e_{i+1}\rangle \otimes |f_{j}\rangle + |e_{i+2}\rangle \otimes |f_{j-1}\rangle \colon \\
&2 \leq i \leq m-2, 2 \leq j \leq n-2, i + j = d + 1\},\\
\mathcal{T}^{(0)} = \mathcal{T}^{(1)} = \mathcal{T}^{(2)} &= \mathcal{T}^{(N-2)} = \mathcal{T}^{(N-1)} = \mathcal{T}^{(N)} = \{0\}
\end{align*}
and $\mathcal{T} =  \bigoplus\limits_{d=0}^N\mathcal{T}^{(d)}$. Similarly, for $2 \leq d \leq N-2$ define
\begin{align*}
\mathcal{S}^{(d)} = \mathrm{span}\{|e_{i}\rangle \otimes |f_{j+2}\rangle &+a|e_{i+1}\rangle \otimes |f_{j+1}\rangle + |e_{i+2}\rangle \otimes |f_{j}\rangle \colon \\
&1 \leq i \leq m-2, 1\leq j \leq n-2, i + j = d\},\\
\mathcal{S}^{(0)} = \mathcal{S}^{(1)} &= \mathcal{S}^{(N-1)} = \mathcal{S}^{(N)} = \{0\}
\end{align*}
and $\mathcal{S} =  \bigoplus\limits_{d=0}^N\mathcal{S}^{(d)}$. Then, 
\begin{enumerate}
\item $\mathcal{T} \subset \mathcal{S}$.
\item
If $|a| \geq 2$, then $\mathcal{S}$ does not contain any nonzero vector of Schmidt rank $\leq 2$ and $\dim \mathcal{S} = (m -2)(n - 2)$.
\item
If $a > 4$, %\textcolor{magenta}{($|a+1|>5$ would give more examples)}, 
then $\mathcal{T}$ does  not contain any nonzero vector of Schmidt rank $\leq 3$ and $\dim \mathcal{T} = (m - 3)(n - 3)$. 
\item
If $a > 4$, then there is at least one vector with Schmidt rank $3$ in $\mathcal{S}\backslash{\mathcal{T}}$.
\end{enumerate}

\end{thm}

\begin{proof}

(1) For any $|e_{i-1}\rangle \otimes |f_{j+2}\rangle + (a + 1)|e_{i}\rangle \otimes |f_{i}\rangle + (a + 1)|e_{i+1}\rangle \otimes |f_{j}\rangle + |e_{i+2}\rangle \otimes |f_{j-1}\rangle \in \mathcal{T}^{(d)}$ 
%$(i + j = d+1)$ 
we take $|e_{i}\rangle \otimes |f_{j+1}\rangle + a|e_{i+1}\rangle \otimes |f_{j}\rangle + |e_{i+2}\rangle \otimes |f_{j-1}\rangle = |e_{i}\rangle \otimes |f_{(j-1)+2}\rangle + a|e_{i+1}\rangle \otimes |f_{(j-1)+1}\rangle + |e_{i+2}\rangle \otimes |f_{j-1}\rangle$ 
%$(i + (j - 1) = d+1-1 = d)$, 
$|e_{i-1}\rangle \otimes |f_{j+2}\rangle + a|e_{i}\rangle \otimes |f_{j+1}\rangle + |e_{i+1}\rangle \otimes |f_{j}\rangle 
= |e_{i-1}\rangle \otimes |f_{j+2}\rangle + a|e_{(i-1)+1}\rangle \otimes |f_{j+1}\rangle + |e_{(i-1)+2}\rangle \otimes |f_{j}\rangle $
%((i -1) + j = d + 1 - 1 = d$)
$\in \mathcal{S}^{(d)}$. 
Then
\begin{align*}
&|e_{i}\rangle \otimes |f_{j+1}\rangle + a|e_{i+1}\rangle \otimes |f_{j}\rangle + |e_{i+2}\rangle \otimes |f_{j-1}\rangle\\ 
&+ 
|e_{i-1}\rangle \otimes |f_{j+2}\rangle + a|e_{i}\rangle \otimes |f_{j+1}\rangle + |e_{i+1}\rangle \otimes |f_{j}\rangle\\
&= |e_{i-1}\rangle \otimes |f_{dj+2}\rangle + (a + 1)|e_{i}\rangle \otimes |f_{j+1}\rangle + (a + 1)|e_{i+1}\rangle \otimes |f_{j}\rangle + |e_{i+2}\rangle \otimes |f_{j-1}\rangle.
\end{align*}
Hence, $\mathcal{T}^{(d)} \subset \mathcal{S}^{(d)}$, and $\mathcal{T} \subset \mathcal{S}$.

(2) Using Proposition~\ref{prp:BD} 

we know that the columns of $G_n(a)$ are linearly independent such that any nonzero linear combination of these columns has at least 3 nonzero entries. Let 
\begin{align*}
C_k = \{|e_{i}\rangle \otimes |f_{j+2}\rangle &+a|e_{i+1}\rangle \otimes |f_{j+1}\rangle + |e_{i+2}\rangle \otimes |f_{j}\rangle \colon \\
&1 \leq i \leq m-2, 1\leq j \leq n-2, i + j = k\}.
\end{align*} 
Then any element of $\phi(C_k)$ is the matrix by replacing the $k$-th anti-diagonal by $(0,\dots,0,1,a,1,0,\dots,0)$. For $2 \leq k \leq N-2$, from the above observation, $\phi(C_k)$ is the set of $[k]-2$ linearly independent, where as in the proof of Theorem~\ref{thm:subspace}, if $m \leq n$, then
$$
[k] = \left\{\begin{array}{ccc}
k+1&\hbox{for}&k \in \{2, 3, \dots, m-2\}\\
%|m-1| = 
m&\hbox{for}&k \in \{m-1, m, \dots, n-1\}\\
n + m - k - 1&\hbox{for} & k \in \{n, n+1, \dots, N-3 = n + m - 4\}
\end{array}.
\right.
$$
Also, it follows that if $M$ is a matrix obtained by taking any nonzero linear combinations of matrices from $\phi(C_k)$, then $M$ has at least 3 nonzero entries in the $k$-th anti-diagonal, 
and the entries of $M$, other than those on $k$-th anti-diagonal are zeros.  Since the determinant of the $3 \times 3$ submatrix with these $3$ nonzero entries in its principal ant-diagonal is clearly nonzero, any linear combination of the matrices in $\phi(C_k)$ has at least one nonzero order-$3$ minor, thus has rank at least $3$. 

Hence the dimension of $\mathcal{S}$ is given by 
\begin{align*}
\sum_{k=2}^{N-2}([k] - 2) &= \sum_{k=2}^{m-2}(k - 1) + \sum_{k= m-1}^{n-1}(m - 2) + \sum_{k=n}^{n+m-4}(n + m - k - 3)\\
&= 2\sum_{k=2}^{m-2}(k - 1) + (m - 2)(n - m +1)\\
&= (m - 2)(n - 2).
\end{align*}
A similar argument holds in the case $m > n$.

(3) Since $a > 4$, $a + 1 > 5$. Hence, it follows from Remark~\ref{rmk:subspace}.

(4) Each generator $|e_{i}\rangle \otimes |f_{j+2}\rangle +a|e_{i+1}\rangle \otimes |f_{j+1}\rangle + |e_{i+2}\rangle \otimes |f_{j}\rangle$ in $\mathcal{S}^{(d)}$ \ $(2 \leq d \leq N-2)$ has Schmidt rank $3$.
\end{proof}

\vskip 3mm

\begin{rmk}
In Theorem~\ref{thm:subspaces} (3) if we replace the condition $a >4$ by the condition $|a+1|>5$, then it would give more examples.
\end{rmk}

\vskip 3mm

\begin{thm}
Let $m$ and $n$ be natural numbers such that $4 \leq \min\{m,n\}$. 
Let $N = n + m -2$, and $\{|e_i\rangle\}_{i=1}^{m}$ (resp. $\{|f_j\rangle\}_{j=1}^{n}$) be the canonical basis for $\C^m$ (resp. $\C^n$)
and set $g(i,j) =|e_{i+1}\rangle \otimes |f_{j+1}\rangle + a |e_{i}\rangle \otimes |f_{j+2}\rangle$.
Consider 
\begin{align*}
\mathcal{S} &= \mathrm{span}\{g(i,j) \colon 1 \leq i \leq m-1, 0 \leq j \leq n -2\},\\
\mathcal{T} &= \mathrm{span}\{g(i,j) + \frac{1}{a} g(i-1,j+1)  \colon 2 \leq i \leq m-1, 0 \leq j \leq n -3\},\\
\mathcal{U} &= \mathrm{span}\{ (g(i,j) + \frac{1}{a} g(i-1,j+1))  + (g(i-1,j+1)+\frac{1}{a} g(i-2, j+2) ) \colon \\
& \qquad 3 \leq i \leq m-1, 0 \leq j \leq n -4\} .
\end{align*}
When $a > 0$ and $a + \dfrac{1}{a} > 4$, we have the similar observation as in Theorem~\ref{thm:subspaces} as follows:

\begin{enumerate}
\item
$\mathcal{U} \subset \mathcal{T} \subset \mathcal{S}$,
\item
 Any element in $\mathcal{S}$ has Schmidt rank $\geq 2$, any generator in $\mathcal{S}$ has Schmidt rank $2$, and $\dim \mathcal{S} = (m-1)(n-1)$,
\item
Any element in $\mathcal{T}$ has Schmidt rank $\geq 3$, any generator in $\mathcal{T}$ has Schmidt rank $3$, and $\dim \mathcal{T} = (m-2)(n-2)$,
\item
 Any element in $\mathcal{U}$ has Schmidt rank $\geq 4$, any generator in $\mathcal{U}$ has Schmidt rank $4$, and $\dim \mathcal{U} = (m-3)(n-3)$.
\end{enumerate}
\end{thm}

\begin{proof}
(1) It is obvious.

(2) As in Theorem~\ref{thm:subspaces} we consider the $(n + 1) \times n$ matrix $K_n(a)$
$$
K_n(a) = \sum_{j=1}^n(|j\rangle +a |j+1\rangle)
% \otimes 
\langle j|.
$$
Then we know that the columns of $K_n(a)$ are linearly independent such that any nonzero linear combination of these columns has at least 2 nonzero entries. Let

\begin{align*}
D_k = \{|e_{i+1}\rangle \otimes |f_{j+1}\rangle &+a|e_{i}\rangle \otimes |f_{j+2}\rangle\\
&1 \leq i \leq m-1, 0\leq j \leq n-2, i + j = k\}.
\end{align*} 
Then any element of $\phi(D_k)$ is the matrix by replacing the $k$-th anti-diagonal by $(0,\dots,0,1,a,0,\dots,0)$. For $1 \leq k \leq N-1$, from the above observation, $\phi(D_k)$ is the set of $[k]-1$ linearly independent, where as in the proof of Theorem~\ref{thm:subspace}, if $m \leq n$, then
$$
[k] = \left\{\begin{array}{ccc}
k+1&\hbox{for}&k \in \{1, 2, \dots, m-1\}\\
m&\hbox{for}&k \in \{m, m+1, \dots, n-1\}\\
n + m - k - 1&\hbox{for} & k \in \{n, n+1, \dots, N-1 = n + m - 3\}
\end{array}
\right.
$$
Also, it follows that if $M$ is a matrix obtained by taking any nonzero linear combinations of matrices from $\phi(D_k)$, then $M$ has at least 2 nonzero entries in the $k$-th anti-diagonal, and the entries of $M$, other than those on $k$-th anti-diagonal are zeros. Since the determinant of the $2 \times 2$ submatrix with these $2$ nonzero entries in its principal anti-diagonal is clearly nonzero, any linear combination of the matrices in $\phi(D_k)$ has at least one nonzero order-$2$ minor, thus has rank at least $2$. 

Hence the dimension of $\mathcal{S}$ is given by 
\begin{align*}
\sum_{k=2}^{N-1}([k] - 1) &= \sum_{k=1}^{m-1}k + \sum_{k= m}^{n-1}(m - 1) + \sum_{k=n}^{n+m-3}(n + m - k - 2)\\
&= \frac{1}{2}(m-1)m + (m - 1)(n - m) + \frac{1}{2}(m-1)(m-2)\\
&= \frac{1}{2}(m-1)(m + 2n - 2m + m-2) = (m-1)(n-1).
\end{align*}
A similar argument holds in the case $m > n$. 

(3) and (4) follow from the same argument in Theorem~\ref{thm:subspace}.
\end{proof}

%%%%%%%%%%%%%%%%%%%%%%%%%%%%%%%%%%%%%%%%%%%%%%%%%%%%
\section{Acknowledgement}
Part of this research was carried out during the fourth author's stay at Department of Mathematics and Informatics, Chiba University. He would like to appreciate this institute for their hospitality. Research of the fourth author is partially supported by JSPS KAKENHI Grant Number JP17K05285.

%%%%%%%%%%%%%%%%%%%%%%%%%%%%%%%%%%%%%%%%%%%%%%%%%%

%%%%%%%%%%%%%%%%%%%%%%%%%%%%%%%%%%%%%%%%%%%%%%%%%%%

\end{document}